\documentclass{birkjour}

\usepackage{amsmath, amsthm, amssymb, fancyhdr}
\usepackage{amsfonts}
\usepackage{ newtxtext,newtxmath}
\usepackage[ansinew]{inputenc}
\usepackage[dvips]{epsfig}
\usepackage{graphicx, tikz}
\usepackage[percent]{overpic}
\usepackage[english]{babel}

\usepackage{cite}
\usepackage{graphicx}
\usepackage{amscd}
\usepackage{xcolor}
\usepackage{bm}
\usepackage{enumerate}

\usepackage{verbatim}
\usepackage{hyperref}
\usepackage{amstext}
\usepackage{latexsym}
\let\oldsqrt\sqrt
\def\sqrt{\mathpalette\DHLhksqrt}
\def\DHLhksqrt#1#2{%
\setbox0=\hbox{$#1\oldsqrt{#2\,}$}\dimen0=\ht0
\advance\dimen0-0.2\ht0
\setbox2=\hbox{\vrule height\ht0 depth -\dimen0}%
{\box0\lower0.4pt\box2}}

\allowdisplaybreaks

\newcommand{\R}{\mathbb{R}} 
\newcommand{\N}{\mathbb{N}} 
\newcommand{\C}{\mathbb{C}} 

\newcommand{\e}{\varepsilon}

\newcommand{\dist}{\textnormal{dist}} 
\newcommand{\diam}{\textnormal{diam}} 
\newcommand{\supp}{\textnormal{supp}} 

\renewcommand{\phi}{\varphi}

\newcommand{\cA}{{\mathcal A}}

\newcommand{\cC}{{\mathcal C}}

\newcommand{\cF}{{\mathcal F}}

\newcommand{\cK}{{\mathcal K}}
\newcommand{\cL}{{\mathcal L}}

\newcommand{\cN}{{\mathcal N}}

\newcommand{\cV}{{\mathcal V}}

\newcommand{\cX}{{\mathcal X}}

\newtheorem{defi}{Definition}[section]
\newtheorem{remark}[defi]{Remark}

\theoremstyle{plain} 
\newtheorem{thm}[defi]{Theorem}

\newtheorem{lemma}[defi]{Lemma}

\newtheorem{athm}{Theorem}

\numberwithin{equation}{section}

\title[A mixed local and nonlocal operator with drift]
 {The principal eigenvalue of  a mixed local and nonlocal operator with drift}

\author[Craig Cowan]{Craig Cowan}

\address{%
Department of Mathematics, University of Manitoba, Winnipeg, MB, Canada}

\email{craig.cowan@umanitoba.ca}

\thanks{M. El Smaily acknowledges partial support from the Natural Sciences and Engineering Research Council of Canada through the NSERC Discovery Grant RGPIN-2017-04313. P. Feulefack is partially funded by a Research Strategic Initiatives Grant (RSIG22) from the University of Northern British Columbia. C. Cowan acknowledges support from the Natural Sciences and Engineering Research Council of Canada through the NSERC Discovery Grant.}

\author{Mohammad El Smaily}
\address{Department of Mathematics and Statistics, University of Northern British Columbia, 		Prince George, BC, Canada}
\email{mohammad.elsmaily@unbc.ca}
\author{Pierre Aime Feulefack }
\address{Department of Mathematics and Statistics, University of Northern British Columbia, 		Prince George, BC, Canada}
\email{PierreAime.Feulefack@unbc.ca}

\begin{document}
\maketitle

\begin{abstract}
We study the eigenvalue problem involving the mixed local-nonlocal operator $ L:= -\Delta  +(-\Delta)^{s}+q\cdot\nabla$~ in a bounded domain  $\Omega\subset\R^N,$ where a Dirichlet condition is posed on $\R^N\setminus\Omega.$ The field $q$ stands for a drift or advection in the medium. We prove the existence of a principal eigenvalue and a principal  eigenfunction for $s\in (0,1/2]$. Moreover, we prove $C^{2,\alpha}$ regularity, up to the boundary, of the solution to the  problem $Lu=f$ when coupled with a Dirichlet condition and $0<s<1/2$.  To prove the regularity and the existence of a principal eigenvalue, we use a continuation argument, Krein-Rutman theorem    as well as a Hopf Lemma and a maximum principle for the operator $L,$ which we  derive  in this paper.

\end{abstract}
{\footnotesize


\vspace{.4cm}
}

\tableofcontents

\section{Introduction and main results }\label{intro}

The study of the  principal eigenvalue of an operator is essential for many important results in the analysis of elliptic and parabolic PDE as well as  the analysis of   elliptic and parabolic intergro-differential equations (IDE). For instance, the principal eigenvalue is fundamental in the study of semi-linear problems \cite{BDVV22b,BM21},  bifurcation theory, stability analysis of equilibrium of reaction-diffusion\cite{BCV16, BHN05}, large deviation principle, and risk-sensitive control \cite{AB22}. The principal eigenvalue of an operator also  plays a role in  determining whether  the  maximum principle holds or not for the operator at hand \cite{BNV94,D06, GM02}. 

 We are interested in the study of  the principal eigenvalue for an operator involving an advection term (or drift) and a mixed local (elliptic) and nonlocal operator. We consider the following problem
\begin{equation}\label{eq1}
	\begin{split}
	\quad\left\{\begin{aligned}
L u=-\Delta  u+(-\Delta)^{s} u+q\cdot\nabla u &= \lambda u && \text{ in \quad $\Omega$}\\
		u &=0     && \text{ on } \quad \R^N\setminus\Omega,
	\end{aligned}\right.
	\end{split}
	\end{equation}
where
\begin{equation}\label{comega}
\Omega \text{ is an open bounded domain of $\R^N$ with $\cC^{2,\alpha}$  boundary.}
\end{equation}  The operator  $L$ is an  elliptic operator (non-self-adjoint)   obtained by the superposition of the classical and the fractional Laplacian $(-\Delta)^s$ where $s\in (0,1/2].$ Problem \eqref{eq1} has also an advection term  $q\cdot \nabla u,$  where
\begin{equation}\label{cq}
 q:\Omega\to \R^N \text{   is a  vector field in the H\"older space } \cC^{0,\alpha}(\overline\Omega).
\end{equation} The vector field $q$ can be viewed as a transport flow  in \eqref{eq1}.

We recall that  the operator $(-\Delta)^s$, $s\in (0,1)$, stands for the fractional Laplacian and it  is defined, for a compactly supported function $u:\R^N\to \R$ of class $\cC^2$, by
\begin{equation*}
\begin{split}
 (-\Delta)^s u(x)&=C_{N,s}\lim_{\varepsilon\to 0^+}\int_{\R^N\setminus B_{\varepsilon}(x)}\frac{u(x)-u(y)}{|x-y|^{N+2s}}\ dy. 
 \end{split}
\end{equation*}
The  constant  $C_{N,s}$ in the above definition  is given by
$$C_{N,s}:= \pi^{-\frac{N}{2}}2^{2s}s\frac{\Gamma(\frac{N}{2}+s)}{\Gamma(1-s)},$$ 
and it is chosen  so that the operator $(-\Delta)^s$  is equivalently defined by its Fourier transform $$\cF((-\Delta)^su)=|\cdot|^{2s}\cF(u).$$  It is  known that we have the following limits
\[
\lim_{s\to 0^+} (-\Delta)^s  u=u \quad\text{ and }\quad\lim_{s\to 1^-}(-\Delta)^su=-\Delta u\quad \text{ for } ~u\in \cC^{2}_c(\R^N).
\]

 \begin{defi}\label{define}By a principal eigenvalue of $L$, we mean a value $\lambda_1\in \R$,   for which  \eqref{eq1} admits a positive solution  $u$ ($u>0$) in $\Omega$. Throughout the paper, we will denote by $ \lambda_1=\lambda_1(\Omega, q)$ the first eigenvalue of $L$ in $\Omega$ subject to Dirichlet condition on $\R^N\setminus\Omega$ and by $\phi_1$ the corresponding unique (up to multiplication by a nonzero real)  eigenfunction with a constant sign over $\Omega.$
\end{defi}
 As in the case of a classical elliptic operator,  we can adopt the following definition for the principal eigenvalue of $L$  (see  \cite{ABR23,BDGQ18, AB22, BNV94,BDR20}  and the references therein): 
\begin{equation}\label{coville}
\begin{split}
\lambda_1:= \sup\left\{
\begin{aligned}
\lambda\in \R 
& \text{ such that  $\exists \phi\in \cC^2(\Omega)\cap\cC^1( \overline\Omega) \cap\cC_c(\R^N),~ ~ \phi>0$~  in ~$\Omega$,  }\\ 
&\text{  \ $\phi=0$~ in ~$\R^N\setminus\Omega$~   satisfying ~ $L\phi\ge \lambda\phi$~ in ~$\Omega$\ }
 \end{aligned}
 \right\}.
\end{split}
\end{equation}
The characterization \eqref{coville} of $\lambda_1$ will become clear from the proofs we provide in this paper.

The interest in the study of problems involving mixed local-nonlocal operator has been growing rapidly in recent years. This is due to their ability to describe the  superstition of two stochastic processes with different scales (Brownian motion and L\'evy process) \cite{DV21}. The mixed local-nonlocal operator in the  form (without advection)
$$\text{$L_0:=-\Delta  +(-\Delta)^{s}$,~~ $s\in (0,1)$,}$$
 has received by far great attention from different points of view. This includes existence and non-existence results \cite{AMT23, BDVV22,BDVV22b,GU22,AR21,G23}, regularity results \cite{SVWZ22,BMS23,M19, FSZ22,GK23}, associated eigenvalue problems \cite{BM21, BDVV23, PFR21,GMV23,CVW20}, and radial symmetry results \cite{BVDV21}.

 In this paper, we consider  a  mixed local-nonlocal operator  with the \emph{additional advection term} $q\cdot\nabla$, where $q\in L^{\infty} (\Omega)$ is a bounded vector field.  We aim to study the existence of  the principal eigenvalue and the corresponding eigenfunction in $\Omega$ for $L:=-\Delta  +(-\Delta)^{s} +q\cdot\nabla$ with $s\in (0,1/2]$. To the best of our knowledge, the presence of an advection term has not been addressed before. 

It is important to note that when the operator $L$ does not include an advection term, that  is $L\equiv L_0$, the operator is self-adjoint and  the study of the principal eigenvalue for~ $L_0$~  relies on a variational characterization via the  Rayleih quotient (see \cite{BDVV23,BDVV22b,CVW20}). Namely,
\begin{equation}\label{Reyleih}
\lambda_1(\Omega): =\inf_{u\in \cX^s_0(\Omega)\setminus\{0\}}\frac{[u]_{\cX^s(\Omega)}}{\|u\|^2_{L^2(\Omega)}},
\end{equation}
where the space $\cX^s_0(\Omega)$ and the semi-norm $[\cdot]_{\cX^s(\Omega)}$ are defined in Section \ref{functional}, below. However, in the presence of advection, the operator $L$ is {\em no longer self-adjoint} and so there is no simple variational formulation for the first eigenvalue  as in \eqref{Reyleih}.    We will prove the existence of such principal eigenvalue of $L$ and the corresponding eigenfunction with the aid of the Krein-Rutman theorem (see \cite{D10}). There are many   versions of the Krein-Rutman theorem in the literature.  We will use that of \cite[Theorem 1.2]{D06} and we recall it  in Theorem \ref{K-R} below.

Lastly, we mention that integro-differential equations  arise naturally in the study of stochastic processes with jumps. They describe a biological species whose individuals diffuse either by a random walk or by a jump process according to  the prescribe probabilities \cite{SVWZ22}. The generator of a L\'evy process has the following general structure
\begin{equation}\label{Operator}
\mathcal{L}u:= \sum_{i,j=1}a_{ij}D_{ij}u +\sum_{j=1}q_{j}D_{j}u +PV.\int_{\R^N}(u(x)-u(y)) K(x-y)\ dy,
\end{equation}
where $K$ is a  measurable kernel on $\R^N$ satisfying $\int_{\R^N}\min\{1,|z|^2\}K(z)\ dz<\infty$ \cite{CVW20}. The first and second terms in \eqref{Operator} correspond to the diffusion and  the drift respectively. The study of the operator $\mathcal{L}$ in \eqref{Operator} with  all its components (diffusion, drift and jump) appears quite intriguing. By far, there are just a few contributions in this direction. In \cite{AB22}, while studying the risk-sensitive control for a class of diffusion with jumps, the authors investigated the existence of the principal eigenvalue for the class of operators $\mathcal{L}$ where the kernel is locally integrable.   In  \cite{ABR23}, the authors also considered a locally integrable kernel  and proved the existence of  generalized principal eigenvalue in $\R^N.$  We refer to \cite[Chap. 3]{GM02} where  elliptic problems involving general second order elliptic  intergro-differential operator have been considered. Note that our operator $L$ in \eqref{eq1}  corresponds to $a_{ij}=\delta_{ij}$ in \eqref{Operator}. In this present work, we only consider   an $L$ where the nonlocal operator is replaced by   the fractional Laplacian.
\medskip

We state our first result  as follows.

\begin{thm}\label{Main-Result-1}
Let  $q\in C^{0,\alpha}(\overline{\Omega})$ for some $0<\alpha<1$  and $s\in (0,1/2]$. Then,  there exists a principal eigenpair $(\lambda_1(\Omega, q), \phi_1)$ for the problem \eqref{eq1} such that
\begin{enumerate}[{\rm \bf (a)}]
\item $\lambda_1(\Omega,q)$ is an eigenvalue of $L$ in $\Omega$ and the corresponding  eigenfunction $\phi_1$ has a constant sign in $\Omega$ and it is unique up to multiplication by a nonzero constant. Moreover, if $s\in (0,1/2)$,  $\phi_1\in \cC^{2,\alpha}(\overline\Omega).$ If $s=\frac{1}{2},$ we have $\phi_1\in \cC^2(\Omega)\cap\cC^{1,\alpha}( \overline\Omega)$.
	\item If $\lambda\in\C$ is an eigenvalue for $L$ in $\Omega$, then we have $\lambda_1(\Omega, q)\le |\lambda|.$
	\item $\lambda_1(\Omega, q)$ can be characterized by the following min-max formula
	\begin{equation}\label{minmax}
	\lambda_1(\Omega,q)=\max_{u\in \cV(\Omega)}\inf_{x\in  \Omega} \frac{Lu}{u},
	\end{equation}
	where 
	\[
\cV(\Omega):=\Big\{u\in \cC^2(\Omega)\cap\cC^1( \overline\Omega) \cap \cC_c(\R^N): \ u>0 \text{ in }{ \Omega} \ \text{ and }\ u\equiv 0 \text{ on }\ \R^N\setminus\Omega\Big\}.
\] 
\end{enumerate}
\end{thm}
We also prove the following Hopf Lemma for the operator $L.$ We emphasize that the result holds for any $s\in(0,1).$ We will use the following result in proving Theorem \ref{Main-Result-1} but we will state the result in a general setting.

\begin{thm}[Hopf Lemma]\label{Hopf2}
Let $s\in (0,1)$. Suppose that $\Omega\subset \R^N$ is a bounded  $C^2$ domain and let  $c_0\in \R$.  Let $u\in \cC^2(\Omega)\cap\cC(\R^N)\cap \cC^1(\overline{\Omega})$ such that $u$ is  bounded in $\R^N$ and 
\begin{equation}\label{eq-Hopf2}
L  u:=-\Delta  u+(-\Delta)^{s} u+q\cdot\nabla u    \ge 0  \text{ in \quad $\Omega$}.
\end{equation}
Let $x_0\in \partial\Omega.$ Assume that $u(x)=c_0$ on $B_{R_0}(x_0)\cap \partial\Omega$, for some $R_0>0,$ and that $u\geq c_0$ in $\R^N.$  If  $u\not\equiv c_0$ in $\R^N,$ then 
\begin{equation}
\partial_{\nu}u(x_0)<0,
\end{equation} 
where $\nu$ denotes the outer unit normal to $\partial\Omega$ at $x_0$.
\end{thm}

\begin{remark}\label{rem.on.Hopf}  Observe that Theorem \ref{Hopf2} holds for all $0<s<1$ and that the function $u$ is not assumed to be $\cC^2(\overline{\Omega}).$ We only assume that $u$ is differentiable at $x_0\in \partial \Omega.$ We will state another version of the Hopf Lemma in Theorem \ref{Hopf} below. However, the  other version requires more regularity on $u$ and its proof, which turns out to be shorter, relies on an inequality satisfied by $(-\Delta)^s u$ in a neighbourhood of $x_0.$ As $\cC^2$ regularity, up to the boundary, is not confirmed for $s=1/2$, we will see that Theorem \ref{Hopf2} turns out to be more helpful, than Theorem \ref{Hopf}, in proving Theorem \ref{Main-Result-1}.
\end{remark}

 The following three theorems address the regularity of solutions to the linear problem
\begin{equation}\label{eq-Regularity-intro}
	\begin{split}
	\quad\left\{\begin{aligned}
		L  u  :=&-\Delta u  +(-\Delta)^{s}u+q\cdot\nabla u &= f && \text{ in \quad $\Omega$,}\\
	u &=0     &&& \text{ on } \quad \R^N\setminus\Omega.
	\end{aligned}\right.
	\end{split}
	\end{equation}
We mention that the regularity of $u$---up to the boundary of $\Omega$, occurs for $0<s<1/2.$ In the case $s=1/2,$ we prove an interior regularity result. 

\begin{thm}[$\cC^{2,\alpha}$ interior regularity when $s=1/2$]\label{Interior-regularity}
Let  $f\in \cC^{0,\alpha}(\overline{\Omega})$ and $q\in \cC^{0,\alpha}(\overline{\Omega}),$ for some $\alpha\in (0,1).$ Assume that $s=\frac{1}{2}$. Let $u\in W^{2,p}(\Omega)$ be a solution of   \eqref{eq-Regularity-intro}, then $u\in \cC^{2,\alpha}(\overline{\Omega'})$ for every $\Omega'\subset \subset \Omega$. Moreover, there exists a constant $C>0$ such that 
\begin{equation}\label{Est-1/2}
\|u\|_{\cC^{2,\alpha}(\overline{\Omega'})}\le C\|f\|_{\cC^{0,\alpha}(\overline\Omega)}
\end{equation}
\end{thm}

The following theorem provides a $W^{2,p}$ estimate for the solution to the mixed local/nonlocal problem \eqref{eq1}.

\begin{thm}\label{strong-solution} 
Let $\Omega\subset \R^N$ be an open bounded set of class $\cC^{1,1}$ with $N>2s$ and $s\in (0,1/2]$. Assume that  $f\in L^p(\Omega)$ with  $1<p<\infty$ and  $q\in L^{\infty}(\Omega).$ Then, the problem \eqref{eq-Regularity-intro}
	has a unique solution  $u\in W^{2,p}(\Omega)$. Furthermore, there exists a constant $C:=C(N,s,p,\Omega)>0$ such that
	\begin{equation}\label{est1}
\|u\|_{W^{2,p}(\Omega)}\le C\|f\|_{L^p(\Omega)}.
   \end{equation}	
\end{thm}
As a consequence of Theorem \ref{strong-solution}, we have the following $\cC^{2,\alpha}$ regularity, up to the boundary,  for \eqref{eq-Regularity-intro} when  $0<s<1/2$.

\begin{thm}[H\"older regularity up to the boundary when $s<1/2$]\label{Holder1} Let $s\in (0,\frac{1}{2})$ and  $ 0<\alpha<1-2s$. Assume that the advection term satisfies $q\in \cC^{0,\alpha}(\overline{\Omega}).$  Then, there is some $C>0$ such that for all $ f \in C^{0,\alpha}(\overline{\Omega})$ there is some $u \in C^{2,\alpha}(\overline{\Omega})$ that satisfies 
\begin{equation}\label{holder}
	\begin{split}
	\quad\left\{\begin{aligned}
		L  u    &= f && \text{ in \quad $\Omega$}\\
	u &=0     && \text{ on } \quad \R^N\setminus\Omega.
	\end{aligned}\right.
	\end{split}
	\end{equation}  Moreover, we have
 $$ \| u \|_{C^{2,\alpha}(\overline{\Omega})} \le C \|f\|_{C^{0,\alpha}(\overline{\Omega})}.$$  
\end{thm}
We briefly comment on the proof of Theorem \ref{Main-Result-1}.  The main ingredient in the  proof of Theorem \ref{Main-Result-1} is   the Krein-Rutman theorem, which relies on  the strong maximum principle for the operator  $L$ and  the $L^p$-theory of the  problem \eqref{eq-Regularity-intro} (see Theorem \ref{strong-solution}).

Indeed, with the aid of the $L^p$-theory of $L_0$ developed in \cite[Theorem 1.4]{SVWZ22} and for more general second order elliptic  intergro-differential operators developed  in \cite[Theorem 3.1.23]{GM02} , we first prove using the method of continuity (see \cite[Theorem 5.2]{GT77}) that for $f\in L^p(\Omega)$, there exists a unique solution  $u\in W^{2,p}(\Omega)$  of  problem \eqref{eq-Regularity-intro}  for any $1<p<\infty$. Then, using the Sobolev (Morrey) embedding  theorem we in fact  have that $u\in \cC^{1,\beta}(\overline\Omega)$,  for any $\beta\in (0,1)$. Moreover,  the extension theorem allows us to extend $u$ by zero to  a $\cC^{0,1}$-function in $\R^N$  and thanks to the  regularity result of \cite[Proposition 2.5]{S07}, we have that $u\in \cC^{2,1-2s}(\overline\Omega)$ for $s\in (0,1/2)$. This allows us to prove  the uniqueness result in Lemma \ref{Uniqueness} and then prove Theorem \ref{strong-solution} for $s\in (0,1/2)$. In the case $s=\frac{1}{2}$, we  obtain an interior $\cC^{2,\alpha}$ regularity for $u$ in Theorem \ref{Interior-regularity}, which also allows us to directly apply the strong maximum principle for $L$ in Theorem \ref{Strong-maximum-principle} and thus complete the  proof of Theorem \ref{strong-solution} for $s=\frac{1}{2}$.


We point out that since our strategy of proving Theorem \ref{Main-Result-1} relies  on   the $L^p$-theory of problem \eqref{eq-Regularity-intro} combined with the application of  the Krein-Rutmen theorem, our result holds for more general mixed local-nonlocal  operators satisfying the strong maximum principle  as  given in \eqref{Operator}  with the kernel of $K$ satisfying 
\begin{equation}\label{c}
\qquad\frac{\kappa_1}{|x-y|^{N+2s}}\le K(x-y)\le \frac{\kappa_2}{|x-y|^{N+2s}},\quad \kappa_1,\kappa_2>0, s\in (0,1/2].
\end{equation}
We refer the interested reader to \cite[Chap. 3]{GM02} for general second order elliptic  intergro-differential operators satisfying such properties. In particular, any nonlocal operator of small order  will satisfy \eqref{c} (see \cite{F23,FJ23} and the references therein).\\

The rest of this paper is organized as follows. In Section \ref{functional},  we introduce some functional spaces. In Section \ref{Preliminary}, we  prove the strong maximum principle and the Hopf Lemma for  $L$. In Section \ref{Lp}, we  develop the $L^p$-theory for $L$ and prove the existence and uniqueness of solution $u\in W^{2,p}(\Omega)$ to problem \eqref{eq-Regularity-intro}. 
Section \ref{Proof-Main} is devoted to the proof of Theorem \ref{Main-Result-1} by using most of the results in the previous sections.

\section{Functional setting}\label{functional}
We start this section by fixing some notation.  
Let $\Omega\subset \R^N$ be an open bounded set. For the vector field $q:\Omega\to \R^N$, we write $q\in L^{\infty}(\Omega)$ (resp. $q\in \cC^{0,\alpha}(\overline{\Omega})$) whenever $q_{j}\in L^{\infty}(\Omega)$ (resp. $q_j\in \cC^{0,\alpha}(\overline{\Omega})$), $j=1,2,\cdots,N.$ We denote  by $\cC^{k,\alpha}(\overline{\Omega})$, $0<\alpha<1$, the Banach space of functions $u\in \cC^k(\overline{\Omega})$ such that derivative of order $k$ belong to $\cC^{0,\alpha}(\overline{\Omega})$ with the norm
\[
\|u\|_{\cC^{k,\alpha}(\overline{\Omega})}:= \|u\|_{\cC^{k}(\overline{\Omega})}+\sum_{|\tau|=k}[D^{\tau}u]_{\cC^{0,\alpha}(\overline{\Omega})},
\]
where
\[
 [u]_{\cC^{0,\alpha}(\overline{\Omega})}=\sup_{x,y\in\overline{\Omega},x\neq y}\frac{|u(x)-u(y)|}{|x-y|^{\alpha}}
\]
and $\cC^{0,\alpha}(\overline{\Omega})$ is the Banach space of functions $u\in \cC^0(\overline{\Omega})$ which are H\"older continuous with exponent $\alpha$ and the  norm  $\|u\|_{\cC^{0,\alpha}(\overline{\Omega})}=\|u\|_{L^{\infty}(\Omega)}+[u]_{\cC^{0,\alpha}(\overline{\Omega})}$.

If $k\in \N$, as usual we set 
$$
W^{k,p}(\Omega):=\Big\{u\in L^p(\Omega)\;:\; \text{ $D^{\alpha}u$ exists for all $\alpha\in \N^{N}$, $|\alpha|\leq k$ and $u\in L^p(\Omega)$ }\Big\}
$$
for the Banach space of $k$-times (weakly) differentialable functions in $L^p(\Omega)$. Moreover, in the fractional setting, for $s\in(0,1)$ and $p\in[1,\infty),$ we set
$$
W^{s,p}(\Omega):=\Big\{u\in L^p(\Omega)\;:\; \frac{u(x)-u(y)}{|x-y|^{\frac{n}{p}+s}}\in L^{p}(\Omega\times \Omega)\Big\}.
$$
The space $W^{s,p}(\Omega)$ is a Banach space with the norm
$$
\|u\|_{W^{s,p}(\Omega)}=\Big(\|u\|_{L^p(\Omega)}^p+\iint_{\Omega\times \Omega}\frac{|u(x)-u(y)|^p}{|x-y|^{n+s p}}\ dxdy\Big)^{\frac{1}{p}}.
$$
We also define the space $\cX^s(\Omega)$ by
\[
\cX^s(\Omega):= \left\{ u\in L^2(\Omega): \quad u|_{\Omega}\in H^1(\Omega);\quad [u]_{\cX^s(\Omega)}<\infty\right\},
\]
where the corresponding Gagliardo seminorm $[\cdot]_{\cX^s(\Omega)}$ is given by
\[
 [u]_{\cX^s(\Omega)}:=\int_{\Omega}|\nabla u|^2\ dx+\int_{\R^N}\int_{\R^N}\frac{|u(x)-u(y)|^2}{|x-y|^{N+2s}}\ dxdy.
\]
Note  that the space $\cX^s(\Omega)$ is a Hilbert space when furnished with the scalar product
\[
\langle u,v\rangle_{\cX^s(\Omega)} := \int_{\Omega}uv\ dx + \int_{\Omega}\nabla u\nabla v\ dx+
\int_{\R^N}\int_{\R^N}\frac{|u(x)-u(y)|^2}{|x-y|^{N+2s}}\ dxdy
\]
and the corresponding norm is given by $\|u\|_{\cX^s(\Omega)}=\sqrt{\langle u,v\rangle_{\cX^s(\Omega)}}$. ~ Define 
\[
\cX_0^s(\Omega):= \big\{ u\in \cX^s(\Omega): \quad u\equiv 0\quad \text{on}\quad \R^N\setminus \Omega\big\}.
\]
Note  that if $u\in\cX^s_0(\Omega)$ then  $u|_{\Omega}\in H^1_0(\Omega)$ due to the regularity assumption of $\partial\Omega$.

Finally,  we define the space $\cL^1_s(\R^N)$ by 
\[
\cL^1_s(\R^N):= \big\{u:\R^N\to \R, \text{ such that $u$ is measurable and }  \|u\|_{\cL_s^1(\R^N)}<\infty\big\},
\]
where
\[
\|u\|_{\cL_1^s(\R^N)}:= \int_{\R^N}\frac{|u(y)|}{1+|y|^{N+2s}} dx.
\]

\section{Hopf lemma, maximum principle, and interior regularity: proof of Theorem \ref{Hopf2} and Theorem \ref{Interior-regularity}}\label{Preliminary}

In this section, we derive some  results  for the operator $L.$ These will be important in the proof of Theorem \ref{Main-Result-1}. 
We start  with  the  following result on the  strong maximum principle for $L$.
\begin{thm}[Strong Maximum Principle]\label{Strong-maximum-principle}
Let $\Omega\subset \R^N$ be an open bounded set and $q\in L^{\infty}(\Omega)$. Let $s\in (0,1)$ and $u\in \cL^1_s(\R^N)$ be a function in $\cC^2(\Omega)\cap\cC(\R^N)$ that satisfies 
\[
\begin{cases}
L  u &  \ge 0~  \text{ in \quad $\Omega$}\\
	u & \ge 0   ~   \text{ on } \quad \R^N\setminus\Omega.
\end{cases}
\]
Then $u>0$ in $\Omega$ or $u\equiv 0$ in $\R^N$.
\end{thm}

\begin{proof}[Proof of Theorem \ref{Strong-maximum-principle}] 
Suppose by contradiction that $u$ is not positive in $\Omega$. Since $\Omega$ is bounded, $\overline{\Omega}$ is compact. Since  $u$ is continuous in $\R^N$ and $u\ge 0$ in $\R^N\setminus\Omega$,  there is a point $x_0\in \Omega$ with
\begin{equation}\label{min}
u(x_0)=\min_{x\in\overline{\Omega}}u(x)\le 0.
\end{equation}
Therefore, since $q$ is bounded, it follows that $q\cdot\nabla u(x_0)=0$ and $\Delta u(x_0)\ge 0$. Hence, from the definition of the operator $L$ we have that
\[
(-\Delta)^su(x_0)\ge 0. 
\]
Whereas by \eqref{min}, we have that $u(x_0)\le u(x)$ for all $x\in \R^N$. It  follows that
\begin{align*}
0\le (-\Delta)^su(x_0)=P.V.\int_{\R^N}\frac{u(x_0)-u(y)}{|x_0-y|^{N+2s}}\ dy=\int_{\R^N}\frac{u(x_0)-u(y)}{|x_0-y|^{N+2s}}\ dy\le 0.
\end{align*}
Moreover, since the integrand  is  non-positive by assumption and \eqref{min}, we conclude that 
\[
u\equiv u(x_0)\quad \text{ in }\quad\R^N.
\]
Now, since $u\ge 0$ in $\R^N\setminus\Omega$,  it follows that  $u\equiv 0$ in $\Omega$ and therefore  $u\equiv 0$ in $\R^N$. This leads to a contradiction and  the proof is established.
\end{proof}

We now prove the Hopf Lemma stated in Theorem \ref{Hopf2}, for all $s\in(0,1).$

\begin{proof}[{\bf{Proof of Theorem \ref{Hopf2}}}]

Let be  $B_r(\bar x)$  a ball centered at $\bar x\in \Omega$ that touches $\partial\Omega$ at $x_0\in \partial\Omega.$ Let $K$ be the set defined by
$$K:=B_r(\bar x)\cap B_{\frac{r}{2}}(x_0).$$
 We introduce the auxiliary function
\[
v(x):= e^{-\alpha \cK(x)}- e^{-\alpha(1+r^2)^{\frac{1}{2}}},\quad \text{ where }\quad \cK(x):= (|x-\bar x|^2+1)^{\frac{1}{2}}
\]
and $\alpha$ is a positive constant to be chosen later. 
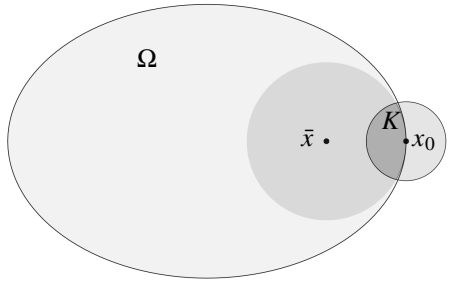
\begin{figure}[ht]
\begin{overpic}[width=0.5\textwidth, tics=10]{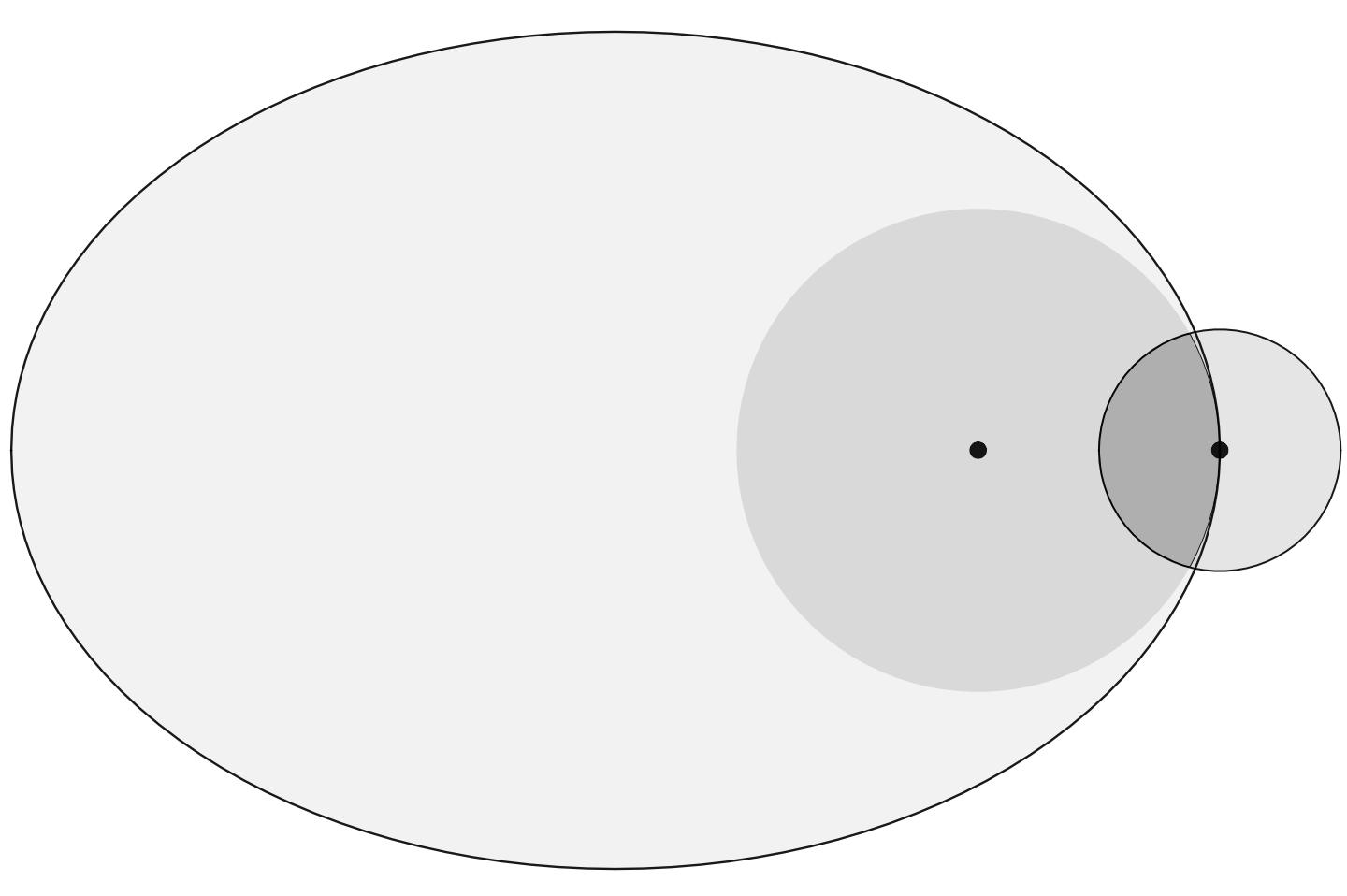}
 \put (30,50) {$ \Omega$}
 \put(67,32){$\bar x$}
  \put(92,32){$x_0$}
  \put(85,36){$K$}
\end{overpic}
\caption{ The open set $K\subset \Omega$ is the intersection of the ball centered at $x_0$ with the ball centered at $\bar x,$ which is tangent to $\partial \Omega$ at $x_0.$ Note that $\overline K\cap \partial \Omega=\{x_0\}$.}
\end{figure}

We have 
\[
v>0\quad\text{in }\quad B_r(\bar x),\quad  v= 0\ \text{ on }\ \partial B_r(\bar x)\quad \text{ and }\quad v<0\ \text{ on } \ \R^N\setminus B_r(\bar x).
\]
For any $x\in\R^N,$ computation shows that 
\begin{equation}\label{Lv}
\begin{split}
L v(x) &=-\Delta v(x) +(-\Delta)^sv(x)+q(x)\cdot \nabla v(x)\\
&= e^{-\alpha\cK(x)}\left ( \frac{\alpha N}{\cK(x)}-\alpha\frac{q(x)\cdot (x-\bar x)}{\cK(x)}-|x-\bar x|^2\Big(\frac{\alpha}{\cK^3(x)}+\frac{\alpha^2}{\cK^2(x)}\Big)\right)\\
&+(-\Delta)^s v(x).
\end{split}
\end{equation}
Observe that $1-e^{-\rho}\le \rho$ for all $\rho\in \R$. Therefore,
\begin{equation}\label{fractional.v}
\begin{array}{l}
\displaystyle{(-\Delta)^sv(x)}
=\displaystyle{\frac{C_{N,s}}{2}\int_{\R^N}\frac{2v(x)-v(x+y)-v(x-y)}{|y|^{N+2s}}\ dy}\vspace{10 pt}\\
= \displaystyle{\frac{C_{N,s} e^{-\alpha \cK(x)}}{2}\int_{\R^N}\frac{2-e^{-\alpha(\cK(x+y)- \cK(x))}-e^{-\alpha(\cK(x-y)- \cK(x))}}{|y|^{N+2s}}\ dy}\vspace{10 pt}\\
\le \displaystyle{\frac{C_{N,s}\alpha e^{-\alpha \cK(x)}}{2}   \int_{\R^N}\frac{(\cK(x+y)- \cK(x))+(\cK(x-y)- \cK(x))}{|y|^{N+2s}}\ dy}\vspace{10 pt}\\
=\displaystyle{ -\alpha e^{-\alpha \cK(x)} (-\Delta)^s\cK(x)}.
\end{array}
\end{equation}
Now, we compute 
\[
\begin{array}{l}
\displaystyle{\frac{2|(-\Delta)^s\cK(x)| }{C_{N,s}}}\vspace{10 pt}\\
\displaystyle{\le \int_{B_1}\frac{|(\cK(x+y)- \cK(x))+(\cK(x-y)- \cK(x))|}{|y|^{N+2s}}\ dy}
\displaystyle{+2\cK(x)\int_{\R^N\setminus B_1}\frac{1 }{|y|^{N+2s}} \ dy}\vspace{10 pt}\\
\displaystyle{\le  \int_{B_1}\int_0^1\frac{|\langle\nabla\cK(x+ty))-\nabla \cK(x-ty),y\rangle|}{|y|^{N+2s}}\ dt dy}+\\
\displaystyle{ 2(\diam(\Omega)^2+1)\int_{\R^N\setminus B_1}\frac{1 }{|y|^{N+2s}}\ dy}\vspace{10 pt}\\
\displaystyle{\le \int_{B_1}\int_0^1\frac{|\nabla\cK(x+ty))-\nabla \cK(x-ty)|}{|y|^{N+2s-1}}\ dt dy+\frac{2\omega_{N-1}(\diam(\Omega)^2+1)}{2s}}.
\end{array}
\]
We also have
\[
\begin{aligned}
|\nabla\cK(x+ty))&-\ \nabla \cK(x-ty)|=\displaystyle{\left|\frac{(x+ty)}{\cK(x+ty))}-\frac{(x-ty)}{\cK(x-ty)}\right|}\vspace{10 pt}\\
&=\displaystyle{\left|ty\left(\frac{1}{\cK(x+ty)}+\frac{1}{\cK(x-ty)}\right)+x\left(\frac{1}{\cK(x+ty)}-\frac{1}{\cK(x-ty)}\right)\right|}\vspace{10 pt}\\
&\le \displaystyle{2t|y|+\left|\frac{1}{\cK(x+ty)}-\frac{1}{\cK(x-ty)}\right||x|}, \text{ as $\cK\geq 1$}\vspace{10 pt}\\
& \le \displaystyle{2t|y|\left(1+\left|\nabla \left(\frac{1}{\cK(z)}\right)\right||x|\right)}, \text{ for some $z$ in the segment $[x-ty, x+ty]$}\vspace{10 pt}\\
&\leq\displaystyle{ 2t|y|\left(1+\left|\frac{z-\bar x}{2{\cK}^3(z)}\right||x|\right)}\leq \displaystyle{2t|y|\left(1+\left|z-\bar x\right||x|\right)} \vspace{10 pt}\\
&\leq \displaystyle{2t|y|\left(1+\left|z- x\right||x|+\left|\bar x\right||x|\right)}\le  \displaystyle{2t|y|\left(1+(3t|y|+|\bar x|)|x|\right)}.
\end{aligned}
\]
Since $x,\bar x\in \Omega$  and $\Omega$ is bounded, we have $|x|,|\bar x|\le D$, where $D$ is a positive constant that depends on $\Omega.$ Also, $|y|<1$ for $y\in B_1$. Therefore, 
\begin{equation}\label{fractional.k}
\begin{array}{l}
{|(-\Delta)^s\cK(x)|  }\leq\displaystyle{\frac{C_{N,s}}{2}\int_{B_1}\int_0^1\frac{|\nabla\cK(x+ty))-\nabla \cK(x-ty)|}{|y|^{N+2s-1}}\ dt dy}~+\vspace{10 pt}\\
\qquad  \qquad \qquad\displaystyle{\frac{2\omega_{N-1}(\diam(\Omega)^2+1)}{2s}}\vspace{10 pt}\\
\le  \displaystyle{\frac{C_{N,s}}{2}\int_{B_1}2t\int_0^1\frac{\left(1+ (3t|y|+|\bar x|)|x|\right)}{|y|^{N+2s-2}}\ dt dy+\frac{2\omega_{N-1}(\diam(\Omega)^2+1)}{2s}}\vspace{10 pt}\\
\le  \displaystyle{\frac{C_{N,s}}{2}\int_{B_1}\frac{2(1+(3+D)D)}{|y|^{N+2s-2}}\  dy+\frac{2\omega_{N-1}(\diam(\Omega)^2+1)}{2s}}\vspace{10 pt}\\
\le \displaystyle{ {C_{N,s}}(1+(3+D)D)\frac{\omega_{N-1}}{2-2s}+\frac{2\omega_{N-1}(\diam(\Omega)^2+1)}{2s}}\vspace{10 pt}\\
:=M.
\end{array}
\end{equation}
We denote the constant obtained in the upper bound of $(-\Delta)^s\cK(x)$ by $M.$  Thus, it follows from \eqref{Lv}, \eqref{fractional.v} and \eqref{fractional.k} that,  for all $x\in \R^N,$
\[
\begin{split}
    L v (x)&\le e^{-\alpha\cK(x)}\left ( \frac{\alpha N}{\cK(x)}-\alpha\frac{q(x)\cdot( x-\bar x )}{\cK(x)}-|x-\bar x|^2\Big(\frac{\alpha}{\cK^3(x)}+\frac{\alpha^2}{\cK^2(x)}\Big) -\alpha M\right).
\end{split}
\]
Now, if $x\in K=B_r(\bar x)\cap B_{\frac{r}{2}}(x_0)$, we have $|x-\bar x|\ge \frac{r}{2}$ and hence
\[
\begin{split}
    L v(x) &\le e^{-\alpha\cK(x)}\left ( \frac{\alpha N}{\cK(x)}+\alpha\frac{\|q\|_{L^{\infty}} |x-\bar x|}{\cK(x)}-\frac{r^2}{4}\Big(\frac{\alpha}{\cK^3(x)}+\frac{\alpha^2}{\cK^2(x)}\Big) -\alpha M\right). 
\end{split}
\]
Since $\cK(x)\geq 1,$ for all $x,$ we can choose $\alpha$ large enough so that 
\[
Lv< 0 \quad\text{ in } \quad B_{\frac{r}{2}}(x_0)\cap \Omega.
\]
Consider the function $w:=-u +\varepsilon  v +c_0,$ where $\varepsilon$ is a positive constant to be chosen later. Since $u\geq c_0$ in $\R^N,$ the maximum principle in Theorem \ref{Strong-maximum-principle}, applied to $u-c_0,$ yields that $u>c_0$ in $\Omega$ (as $u\not\equiv c_0$ in $\R^N$). As the set $\overline B_{r}(\bar x)\setminus K$ is compact,  the minimum of $u$ over $\overline B_{r}(\bar x)\setminus K$ is attained. So this minimum will be strictly greater than $c_0$. That is, 
 $$\min_{\overline B_{r}(\bar x)\setminus K} u(x)>c_0.$$ 
 Thus, we can find  a constant $\delta>0$ such that 
 \[
u\ge c_o+\delta \ \text{ in }\quad \overline B_{r}(\bar x)\setminus K.
 \]
Then, for $x\in \overline B_{r}(\bar x)\setminus K$
\[
w(x)\le -\delta+\varepsilon v=-\delta+\varepsilon\big( e^{-\alpha \cK(x)}- e^{-\alpha(1+r^2)^{\frac{1}{2}}}\big)\le -\delta+\varepsilon\big( 1- e^{-\alpha(1+r^2)^{\frac{1}{2}}}\big).
\]
Chosen $\varepsilon$ sufficiently small, say
\[
\varepsilon< \frac{\delta}{1- e^{-\alpha(1+r^2)^{\frac{1}{2}}}},
\]
we have 
\[
w<0\quad \text{ in }\quad \overline B_{r}(\bar x)\setminus K.
\]
Since $u\ge c_0$  and $v<0$ in $\R^N\setminus  B_{r}(\bar x)$, we have that  $w<0$ on $\R^N\setminus K$. We also have 
\[
w<0\quad \text{ in } \ \partial K\setminus \{x_0\} \ \text{ and } \ w(x_0) =0.
\]
 Moreover,
\[
Lw= L(-u+\varepsilon v+c_0)\le \varepsilon Lv<0 \quad \text{ in }\quad K.
\]
From the maximum principle, applied to $w$, we obtain $w<0$ in $K$ (since $w\not\equiv 0$ in $\R^N$). As $w(x_0)=0,$ it follows that the maximum of $w$ over $\overline{K}$ is attained at $x_0.$
Therefore, the normal derivative $\partial_\nu w$ satisfies
\begin{equation}\label{d-derivative}
\partial_{\nu}w(x_0)=-\partial_{\nu}u(x_0)+\varepsilon\partial_{\nu}v(x_0)\ge 0.
\end{equation}
We compute now the normal derivative of $v$ over $\partial B_r(\bar x). $ We have 
\[ \nabla v(x)= -\alpha \frac{(x-\bar x)}{\cK(x)}e^{-\alpha \cK(x)}, \quad\text{ for all } x\in \Omega.\]
Thus, for $x\in \partial B_r(\bar x),$ we have 
\[\partial_\nu v(x)=(x-\bar x)\cdot \nabla v(x)= -\alpha \frac{|x-\bar x|}{\cK(x)}e^{-\alpha \cK(x)}<0.\]
In particular, $\partial_\nu v(x_0)<0$ and it follows from \eqref{d-derivative} that 
\[
\partial_\nu u(x_0)\leq \varepsilon\partial_\nu v(x_0)<0.
\]
This completes the proof.
\end{proof}
We state and prove another version of the Hopf Lemma in the next theorem. We refer the reader to Remark \ref{rem.on.Hopf} above for more details on the difference between Theorem \ref{Hopf} and Theorem \ref{Hopf2}.
\begin{thm}[Hopf Lemma]\label{Hopf}
Suppose that $\Omega\subset \R^N$ is a bounded  $C^2$ domain and let  $c_0\in \R$.  Let $u\in \cC^2(\overline\Omega)\cap\cC(\R^N)$ such that $u$ is  bounded in $\R^N$ and 
\begin{equation}\label{eq-Hopf}
L  u   \ge 0  \text{ in \quad $\Omega$}.
\end{equation}
Let $x_0\in \partial\Omega$. Assume that $u(x)=c_0$ on $B_{R_0}(x_0)\cap \partial\Omega$, for some $R_0>0,$ and that $u\geq c_0$ in $\R^N.$  If  $u\not\equiv c_0$ in $\R^N,$ then 
\begin{equation}
\partial_{\nu}u(x_0)<0,
\end{equation} 
where $\nu$ denotes the outer unit normal to $\partial\Omega$ at $x_0$.
\end{thm}
\begin{proof}[{\bf Proof of Theorem \ref{Hopf}}]
 The proof mostly relies on the fact that 
 \begin{equation}\label{F-Laplacian-less}
 (-\Delta)^su\le 0 \text{ in }B_\rho(x_0)\cap \Omega,\text{ for some $0<\rho<R_0,$}
 \end{equation}
and the Hopf lemma for elliptic operators. The proof of inequality \eqref{F-Laplacian-less} is done in details in \cite[inequality (2.9) in the proof of Theorem 2.9]{BDVV23} and we will omit it here. 

We note that if there is a point $y\in B_\rho(x_0)$ such that $u(y)=c_0,$ we apply the maximum principle in Theorem \ref{Strong-maximum-principle} to \eqref{eq-Hopf} (knowing that $u\geq c_0$ in $\R^N$) and conclude  that $u\equiv c_0$ in $\R^N, $ which is a contradiction.

Now, we combine \eqref{eq-Hopf} and \eqref{F-Laplacian-less} to deduce that there exists $\rho>0$ such that
 \begin{equation}\label{elliptic}
 0\le Lu\le -\Delta u +q\cdot \nabla u\quad \text{ in }\quad B_{\rho}(x_0)\cap \Omega, 
 \end{equation}
and $u\geq c_0$ in $\R^N.$  The elliptic maximum principle (see \cite[Lemma 3.4]{GT77}, for e.g.) implies that either $u\equiv c_0$ in $B_\rho (x_0)\cap \Omega$ (this cannot happen because of the note above) or $u>c_0$ in $B_\rho(x_0)\cap \Omega.$ Moreover, the Hopf Lemma for elliptic operators (here, we have $-\Delta+q\cdot \nabla$) implies that $\partial_\nu u (x)<0$ for all $x \in B_{\rho}(x_0)\cap \partial\Omega,$ which is part of  $\partial(B_{\rho}(x_0)\cap \Omega) $. In particular, we have 
$ \partial_{\nu}u(x_0)<0$
and  this completes the proof.
\end{proof}

Next, we give the proof of the interior regularity when $s=1/2.$ We mention that, for $0<s<1/2,$ we will have regularity   $C^{2,\alpha}(\overline \Omega)$ up to the boundary. The latter is done in Theorem \ref{Holder1}.

\begin{proof}[{\bf Proof of Theorem \ref{Interior-regularity}}] Let $f\in \cC^{0,\alpha}(\overline\Omega)$. 
Let  $\Omega'$ and $\Omega_1$ be two open subsets of $\Omega\subset \R^N$ such that 
$$\Omega'\subset\subset\Omega_1\subset \subset \Omega.$$ 
Define the cut-off function $\eta\in \cC_c^{\infty}(\Omega_1)$ as
\begin{equation}\label{cut-off}
\eta(x)=1\quad\text{ if }\quad x\in \Omega'\qquad\text{and}\qquad
\eta(x)= 0\quad\text{ if }\quad x\in \R^N\setminus \Omega_1,
\end{equation}
such that $0\le \eta(x)\le 1$ for all $ x\in \R^N$ and there exists $C_1,C_1>0$ such that
\[
|D\eta|<\frac{C_1}{\dist(\Omega',\partial\Omega)}\qquad\text{and}\qquad |D^2\eta|<\frac{C_2}{(\dist(\Omega',\partial\Omega))^2}.
\]
We set 
\[
v:=\eta u\quad\text{ and }\quad w:=(1-\eta)u
\]
Since $u\in W^{2,p}(\Omega)$, it holds that  $u\in \cC^{1,\alpha}(\overline{\Omega})$ applying the Bootstrap argument as in Lemma \ref{Uniqueness} with $\beta=\alpha$. Then, we  compute
\begin{align*}
-\Delta v&=-\eta\big((-\Delta)^{\frac{1}{2}}u +q\cdot \nabla u - f\big)-2\nabla u\cdot \nabla \eta-u\Delta \eta\\
&:=\widetilde f.
\end{align*}
We note that all elements of $\widetilde f$  are supported in $\Omega_1$.  We need to show that $\widetilde f\in \cC^{0,\alpha}(\R^N)$.  To do so, we only need to show that $\eta(-\Delta)^{\frac{1}{2}}u\in \cC^{0,\alpha}(\R^N)$ since other terms  follow easily. We write 
$$\eta(-\Delta)^su= (-\Delta)^sv -u(-\Delta)^s\eta +I(u,\eta)$$
where
$$
I(u,\eta)(x):= \int_{
\R^N}\frac{(u(x)-u(x+z))(\eta(x)-\eta(x+z))}{|z|^{N+1}}\ dz.
$$
 Since $\supp \ v\subset \Omega$, we have 
\begin{equation}
\|v\|_{\cC^{1,\alpha}(\R^N)}= \|v\|_{\cC^{1,\alpha}(\overline{\Omega})} \le C\|u\|_{\cC^{1,\alpha}(\overline{\Omega})}.
\end{equation}
Then, we  use the regularity result of \cite[Proposition 2.7]{S07} (with $s=\frac{1}{2}$ and $l=0$) to get  $(-\Delta)^{\frac{1}{2}}v\in\cC^{0,\alpha}(\R^N) $ and 
\[
\|(-\Delta)^{\frac{1}{2}}v\|_{\cC^{0,\alpha}(\R^N)}\le C \|v\|_{\cC^{1,\alpha}(\R^N)} \le C \|u\|_{\cC^{1,\alpha}(\overline{\Omega})}.
\]
Also, $\|u(-\Delta)^{\frac{1}{2}}\eta\|_{\cC^{0,\alpha}(\R^N)}\le C\|(-\Delta)^{\frac{1}{2}}\eta\|_{L^{\infty}(\R^N)}\|u\|_{\cC^{0,\alpha}(\R^N)}\le  C \|u\|_{\cC^{1,\alpha}(\overline{\Omega})}$.

We now show that $I(u,\eta)\in \cC^{0,\alpha}(\R^N)$. Note that   since $\Omega_1$ is bounded, we  let $B_{R_0}$  be a ball centred at zero with radius $R_0>0$ and  containing  $\Omega_1 $. We set $R:=R_0+|x|+1$  for any fixed $ x\in \Omega_1$. Observe that if $|z|\ge R$, then $|x+ z|\ge |z|-|x|\ge R-|x|=R_0+1>R_0$. Therefore, $\eta(x+z)\equiv 0$ on $\R^N\setminus B_R$. Next, for $x_1, x_2\in \Omega_1$,  we write
\[
|I(u,\eta)(x_1)-I(u,\eta)(x_2) |\le|I_1|+|I_2|,
\]
where
\[
I_1:= \int_{
B_R}\frac{\sum_{k=1}^{2}(-1)^{k-1}[(u(x_k)-u(x_k+z))(\eta(x_k)-\eta(x_k+z))]}{|z|^{N+1}} dz 
\]
and 
\[
I_2:=\int_{
\R^N\setminus B_R}\frac{\sum_{k=1}^{2}(-1)^{k-1}[(u(x_k)-u(x_k+z))\eta(x_k)]}{|z|^{N+1}}\ dz 
\]
We estimate the integrand of $I_1$ using the fundamental theorem of calculus as follows,
\begin{align*}
&\Big|\sum_{k=1}^{2}(-1)^{k-1}[(u(x_k)-u(x_k+z))(\eta(x_k)-\eta(x_k+z))]\Big|\\
&\qquad\quad=\Big|\int_0^1\langle \nabla u(x_1+tz)-\nabla u(x_2+tz),z\rangle\ dt(\eta(x_1)-\eta(x_1+z))\\
&\qquad\qquad\qquad+\int_0^1\langle \nabla \eta(x_1+\tau z)-\nabla \eta(x_2+\tau z),z\rangle\ d\tau (u(x_2)-u(x_2+z))\big|\\
&\qquad\quad\le C |z|^2\int_0^1| \nabla u(x_1+tz)-\nabla u(x_2+tz)|\ dt\\
&\qquad\qquad\qquad+C\| u\|_{\cC^{1}{(\overline\Omega})}|z|^2\int_0^1| \nabla \eta(x_1+\tau z)-\nabla \eta(x_2+\tau z)|\ d\tau\\
&\qquad\quad\le C\big(\|u\|_{\cC^{1,\alpha}(\overline{\Omega})}+
\| u\|_{\cC^{1}{(\overline\Omega})}\big)|z|^2 |x_1-x_2|^{\alpha}\le C\| u\|_{\cC^{1,\alpha}(\overline\Omega)}|z|^2 |x_1-x_2|^{\alpha}.
\end{align*}
Consequently, 
\[
|I_1|  \le C\| u\|_{\cC^{1,\alpha}(\overline\Omega)} |x_1-x_2|^{\alpha}\int_{B_R}\frac{|z|^2}{|z|^{N+1}}\ dz\le C\| u\|_{\cC^{1,\alpha}(\overline\Omega)} |x_1-x_2|^{\alpha}.
\]
We now estimate $I_2$. Observe first that 
\[
\begin{split}
&\sum_{k=1}^{2}(-1)^{k-1}[(u(x_k)-u(x_k+z))\eta(x_k)]\\
&\qquad=[v(x_1)-v(x_2)]-[u(x_1+z)-u(x_2+z)]\eta(x_1)+ [\eta(x_1)-\eta(x_2)]u(x_2+z).
\end{split}
\]
Therefore, we have
\[
\begin{split}
|I_2|&\le C\big(\| v\|_{\cC^{1}(\overline\Omega)} +\| u\|_{\cC^{1}(\overline\Omega)} +\| u\|_{L^{\infty}(\Omega)}\big) |x_1-x_2|\int_{\R^N\setminus B_R}\frac{dz}{|z|^{N+1}}\\
&\le C\| u\|_{\cC^{1,\alpha}(\overline\Omega)} |x_1-x_2|.
\end{split}
\]
Since it not difficult to see that $I(u,\eta)\in L^{\infty}(\R^N)$, we conclude that $I(u,\eta)\in \cC^{0,\alpha}(\R^N)$.
We then consider  the equation
\[
-\Delta v= \widetilde f\quad\text{ in} \quad \R^N.
\]
Since $ \widetilde f\in \cC^{0,\alpha}(\R^N)$, we can apply the regularity theory for  classical elliptic PDEs to see that
\[
v\in \cC^{2,\alpha}(\overline{\Omega'}).
\]
Since $v = u$ in $\Omega'$ and since  $\Omega'$ was arbitrary, the  proof of Theorem \ref{Interior-regularity} is complete.
\end{proof}

\section{\texorpdfstring{$L^p$}{} theory and regularity up to the boundary: proof of Theorem \ref{strong-solution} and Theorem \ref{Holder1}}\label{Lp}

This section is dedicated to the $L^p$-theory of the operator $L$ and to the $\cC^{2,\alpha}(\overline{\Omega})$ regularity. We will first prove  the following  problem 
\begin{equation}\label{eq-Regularity}
	\begin{split}
	\quad\left\{\begin{aligned}
		L  u  :=&-\Delta u  +(-\Delta)^{s}u+q\cdot\nabla u &= f && \text{ in \quad $\Omega$,}\\
	u &=0     &&& \text{ on } \quad \R^N\setminus\Omega.
	\end{aligned}\right.
	\end{split}
	\end{equation}
	has a unique solution a $u\in W^{2,p}(\Omega)$ (see Theorem \ref{strong-solution}).
 This  extends the  $W^{2,p} $ estimate  done in \cite{SVWZ22} for $L_0$ to the case where an advection term is present in the equation.  We will need the following lemma.

\begin{lemma}\label{Uniqueness} Assume that $0<s<1/2,$  $ 1<p<\infty$ and that the advection term satisfies $q\in \cC^{0,\alpha}(\overline{\Omega})$ for some $0<\alpha<1.$ Let $u \in W^{2,p}(\Omega)$ be a solution of 
\begin{equation}\label{kernel}
	\begin{split}
	\quad\left\{\begin{aligned}
		L  u    &= 0 && \text{ in \quad $\Omega$}\\
	u &=0     && \text{ on } \quad \R^N\setminus\Omega.
	\end{aligned}\right.
	\end{split}
	\end{equation}  Then, $u=0$. 
\end{lemma}

\begin{proof}[{\bf Proof of Lemma \ref{Uniqueness}}]  We need to prove the solution is sufficiently regular first (so that we use the strong maximum principle, stated in Theorem \ref{Strong-maximum-principle}).  

First, consider the case of $ p \geq N$ and then note that we have $ L_0u= -q \cdot \nabla u.$  The right hand side is in $L^T(\Omega)$ for all $T<\infty$ and we can then apply the $L^p$ theory   for the operator  $L_0$ in \cite{SVWZ22} to see that $ u \in W^{2,T}(\Omega)$ for all $T<\infty$ and hence $ u \in C^{1,\beta}(\overline{\Omega})$ for all $ 0<\beta<1$. Now, since  we have $ u \in C^{1,\beta}( \overline{\Omega}),$ and thanks to the regularity assumption on the boundary of $\Omega$,  we can extend $u$ by zero outside  $\Omega$ and still  denote by $u$ (see \cite[Lemma 6.37]{GT77}) and get that  the extension is a $C^{0,1}( \R^N)$ function. We can then apply the regularity result of \cite[Proposition 2.5]{S07}  to see that
$$ g:=(-\Delta )^s u \in C^{0,1-2s}(\R^N).$$   
Now, we can write the equation as $ -\Delta u = -g - q \cdot \nabla u $ in $\Omega$ with $u=0$ on $ \partial \Omega.$ Hence, as $q\in \cC^{0,\alpha}(\overline{\Omega}),$ the right hand side $-g - q \cdot \nabla u$ is a H\"older function. Thus, $u \in C^{2, 1-2s}( \overline{\Omega})$ from the classical theory of elliptic PDEs. We can then apply the maximum principle to get that $u\equiv 0$ in $\R^N$.

Second, we suppose $ 1<p<N$ and set $t_1:=\frac{Np}{N-p}$.   Then we have 
$$ u \in W^{2,p}(\Omega) \subset W^{1, t_1}(\Omega)$$
by the Sobolev embedding theorem. Hence, as $ L_0u= -q \cdot \nabla u,$ the $L^p$ theory for the operator  $L_0$ in \cite{SVWZ22} yields that $ u \in W^{2, t_1}(\Omega)$. Again, the Sobolev embedding theorem  implies that $u\in W^{1,t_2}(\Omega)$, where $t_2:=\frac{Nt_1}{N-t_1}>t_1$. If $t_2<N$, we can do this a finite number of times until we get $u \in W^{2,t}(\Omega)$ for some $t>N.$ At this stage, we become in the  setting of the first case. The proof of Lemma \ref{Uniqueness} is complete.
\end{proof}

We now have all what is needed to prove Theorem \ref{strong-solution}, which we do as follows.

\begin{proof}[{\bf Proof of Theorem \ref{strong-solution}}]
We apply the method of continuity.  To ease the notation, 
 we define the operator
\[
 L_0u:= -\Delta u+(-\Delta)^su,
\]
and for $\lambda\in\R$, we  consider the family of operators
\[
L_{\lambda}u\equiv (1-\lambda)L_0u +\lambda Lu= L_0+\lambda q\cdot\nabla u.
\]
Next, for $u\in W^{2,p}(\Omega)$ and $\lambda\in\R$,  consider the problem 
\begin{equation}\label{eq-existence}
	\begin{split}
	\quad\left\{\begin{aligned}
		L_{\lambda}  u    &= f && \text{ in \quad $\Omega$}\\
	u &=0     && \text{ on } \quad \R^N\setminus\Omega.
	\end{aligned}\right.
	\end{split}
	\end{equation}
Let $\cA$ be the set given by
\begin{equation}\label{setA}
\cA:=\left\{
\begin{aligned}
\lambda\in [0,1]: ~\exists C_{\lambda}>0 \text{ such that for all  }f\in L^p(\Omega),  \eqref{eq-existence} \text{ has a  }\\
 \text{ solution $u$ such that }\|u\|_{W^{2,p}(\Omega)}\le C_{\lambda}\|f\|_{L^p(\Omega)}\qquad\qquad
 \end{aligned}
 \right\}.
\end{equation}
In \eqref{setA}, we take the constant $C_{\lambda}$ to be the smallest constant such that $\|u\|_{W^{2,p}(\Omega)}\le C_{\lambda}\|f\|_{L^p(\Omega)}$ holds. In other words, if $C_{\lambda}>\varepsilon>0$ then there exists  $f_{\varepsilon}\in L^p(\Omega)$ such that 
\begin{equation}\label{Smallest}
\|u\|_{W^{2,p}(\Omega)}\ge (C_{\lambda}-\varepsilon)\|f_{\varepsilon}\|_{L^p(\Omega)}.
\end{equation}

Note that $\cA$ is not empty since we have that $0\in \cA$ by \cite[Theorem 1.4]{SVWZ22}. Therefore, we only need to show that $1\in \cA$.  To do that, it suffices to prove that $\cA$ is both open and closed in $[0,1].$ More precisely, it suffices to prove that for any fixed $\lambda_0\in \cA$ and $f\in L^p(\Omega)$, there is an $\varepsilon>0$ such that $\lambda_0\pm\varepsilon\in \cA$ and that 
 any bounded sequence $\{\lambda_n\}_n\subset \cA$ has a convergence subsequence.

\paragraph{{\bf $\cA$ is open.}} We fix $\lambda_0\in \cA.$  We look for a solution  $u\in W^{2,p}(\Omega)$ of problem \eqref{eq-existence}  in the form $u=v_0+\Phi$, where $v_0$ solves   \eqref{eq-existence} with $\lambda=\lambda_0.$  For  $\varepsilon\in\R$, we  introduce the operator $\cN_{\varepsilon}$ given by
$$\Psi = \cN_{\varepsilon}(\Phi)$$
where $\Psi$  solves the equation
$$
L_{\lambda_0} \Psi = \pm\varepsilon\left( q\cdot\nabla v_0+ q\cdot\nabla \Phi\right).
$$
The operator  $\cN_{\varepsilon}$ maps $W^{2,p}(\Omega)$ into itself.
We claim that if $\varepsilon$  is chosen appropriately,  then $\cN_{\varepsilon}$ is a contraction in $W^{2,p}(\Omega)$. Indeed, since $\lambda_0\in A$ there exists a constant  $C_{\lambda_0}>0$ such that 
\[
\|\cN_{\varepsilon}(\Phi)\|_{W^{2,p}(\Omega)}=\|\Psi\|_{W^{2,p}(\Omega)}\le C_{\lambda_0}\|\pm\varepsilon\left( q\cdot\nabla v_0+ q\cdot\nabla \Phi\right)\|_{L^{p}(\Omega)}.
\]
Now, let $\Phi_1$ and $\Phi_2$ be taken in $W^{2,p}(\Omega)$. Then
\begin{align*}
\|\cN_{\varepsilon}(\Phi_1)-\cN_{\varepsilon}(\Phi_2)\|_{W^{2,p}(\Omega)}&=\|\Psi_1-\Psi_2\|_{W^{2,p}(\Omega)}\\&\le |\varepsilon| C_{\lambda_0}\|q\|_{L^{\infty}(\Omega)}\|\nabla (\Phi_1-\Phi_2)\|_{L^{p}(\Omega)}\\
&\le |\varepsilon| C_{\lambda_0}C\|q\|_{L^{\infty}(\Omega)}\|\Phi_1-\Phi_2\|_{W^{2,p}(\Omega)},
\end{align*}
where the constant $C$ and $C_{\lambda_0}$ are independent of $\varepsilon$ and $\Phi$. Taking $\varepsilon$ such that $$|\varepsilon|\le\frac{1}{2C_{\lambda_0}C\|q\|_{L^{\infty}(\Omega)}},$$ we get that $\cN_{\varepsilon}$ is a contraction mapping. By the fixed point theorem,  for each such $\varepsilon$ there exists a fixed point $\Phi$ such that $\cN_{\varepsilon}(\Phi)=\Phi$. We just showed that the equation  $L_{\lambda_0\pm\varepsilon}u=f$ has a solution $u\in W^{2,p}(\Omega).$ Moreover, 
from the definition of $\cN_{\varepsilon}$ and the choice of $\varepsilon$ above, it follows that
\begin{align*}
\|\Phi\|_{W^{2,p}(\Omega)}=\|\cN_{\varepsilon}(\Phi)\|_{W^{2,p}(\Omega)}&\le  |\varepsilon| C_{\lambda_0}C\|q\|_{L^{\infty}(\Omega)}\left(\|\Phi\|_{W^{2,p}(\Omega)} + \|v_0\|_{W^{2,p}(\Omega)}\right)\\
& \le \frac{1}{2}\left(\|\Phi\|_{W^{2,p}(\Omega)} + \|v_0\|_{W^{2,p}(\Omega)}\right)
\end{align*}
so that $\|\Phi\|_{W^{2,p}(\Omega)}\le 2\|v_0\|_{W^{2,p}(\Omega)}$ and  the norm of $u$ in $W^{2,p}(\Omega)$ becomes
\begin{align*}
\|u\|_{W^{2,p}(\Omega)}&\le \big(\|v_0\|_{W^{2,p}(\Omega)} +\|\Phi\|_{W^{2,p}(\Omega)}\big)\\
&\le C_2\|v_0\|_{W^{2,p}(\Omega)}\\
&\le C\|f\|_{L^p(\Omega)}.
\end{align*}
Therefore, $\lambda_0\pm \varepsilon\in \cA.$\\

\paragraph{\bf $\cA$ is closed.} In other to complete the proof of  theorem, we show that $\cA$ is closed. Let then $\{\lambda_n\}_n\subset \cA$ be a sequence  such that $\lambda_n\to \lambda_0\in \R$ as $n\to \infty$. We claim that $\lambda_0\in \cA$.  Let $f\in L^p(\Omega)$. Since $\lambda_n\in \cA$ for any $n$, there exists $u_n$ that satisfies   \eqref{eq-existence}  with $\lambda_n$ in place of $\lambda$ and
\[
 \|u_n\|_{W^{2,p}(\Omega)}\le C_{n}\|f\|_{L^p(\Omega)}
\]
where $C_n:=C_{\lambda_n}$. Therefore, if $\{C_n\}_n$ is  bounded in $n$ then, the sequence $\{u_n\}_n$ is also uniformly bounded in $W^{2,p}(\Omega)$ so that passing to a subsequence, we have 
\[
u_n\rightharpoonup u\ \text{ weakly \  in } \ W^{2,p}(\Omega)\quad \text{ and }\quad u_n \to u\ \text{ \ strongly\  in }\ W^{1,p}(\Omega)
\] 
with $u$ satisfying the problem \eqref{eq-existence} with $\lambda_0$ in place of $\lambda$ and 
\[
\|u\|_{W^{2,p}(\Omega)}\le \liminf_{n\to \infty}\|u_n\|_{W^{2,p}(\Omega)}\le C\|f\|_{L^p(\Omega)}.
\]
This shows that $\lambda_0\in \cA$ and ends the proof in the case where $\{C_n\}_n$ is bounded.

Indeed, we will show next that the only possible case. Assume to the contrary that 
$\{C_n\}_n$ is unbounded. Then, passing  to a subsequence,  we may have that 
 $C_n\to \infty$ as $n\to \infty.$ Thus, there exists a sequence $\{f_n\}_n$ such that  for large $n$ we have 
\[
 \|u_n\|_{W^{2,p}(\Omega)}\ge (C_{n}-1)\|f_n\|_{L^p(\Omega)}.
\]
Note that the above inequality holds since from \eqref{Smallest} in which the constant $C_n:=C_{\lambda_n}$ is the smallest constant such that  $\|u\|_{W^{2,p}(\Omega)}\le C_n\|f\|_{L^p(\Omega)}$ holds. Let 
$$t_n:= \|u_n\|_{W^{2,p}(\Omega)}, \qquad v_n:= \frac{u_n}{t_n},\qquad\widetilde f_n:= \frac{f_n}{t_n}.$$
Then $\|v_n\|_{W^{2,p}(\Omega)}=1$,  $\widetilde f_n\to 0$ in $L^p(\Omega)$ as $n\to \infty$ and $v_n$ satisfies the equation
 \begin{equation}\label{eq-existence-n}
L_{0}v_n= -\lambda_n q\cdot\nabla v_n +\widetilde f_n\quad\text{ in }\quad\Omega, \qquad v_n=0 \ \text{ on }\ \R^N\setminus\Omega.
\end{equation}
It follows from \cite[Theorem 1.4]{SVWZ22}  that
\begin{align*}
\|v_n\|_{W^{2,p}(\Omega)}&\le C\left(\|\widetilde f_n-\lambda_n q\cdot\nabla v_n\|_{L^p(\Omega)}+\|v_n\|_{L^p(\Omega)}\right)\\
&\le C\left(\|\widetilde f_n\|_{L^p(\Omega)}+\|q\|_{L^{\infty}(\Omega)}\| v_n\|_{W^{1,p}(\Omega)}+\|v_n\|_{L^p(\Omega)}\right)\\
&\le C\left(K+\|q\|_{L^{\infty}(\Omega)}+1\right).
\end{align*}
 This shows that the sequence $\{v_n\}_n$ is uniformly  bounded in $W^{2,p}(\Omega)$  so that the Banach-Alaoglu Theorem implies the existence of a subsequence, which we still label as $\{v_n\}_n$, that converges weakly to some $v_0\in W^{2,p}(\Omega)$ and strongly in $W^{1,p}(\Omega)$ thanks to the compactness of the  Sobolev embedding.  Hence $\|v_0\|_{W^{2,p}(\Omega)}\le 1$.  We now consider two possibilities.

 If $v_0=0$, this will  contradict  the normalization  $\|v_n\|_{W^{2,p}(\Omega)}=1$: indeed, if $v_0=0$, then
 \[
  \text{$\lambda_{n}q\cdot\nabla v_n\to 0$\quad  strongly\  in\  $L^p(\Omega)$}
\]
  thanks to the compactness of the  Sobolev embedding. Hence, we can use the $L^p$ theory for $L_0$  and \eqref{eq-existence-n} to see that  $v_n\to 0$ in $W^{2,p}(\Omega),$  which contradicts the normalization of $v_n$ in $W^{2,p}(\Omega)$. 
  
  The other possibility is that  $v_0\neq 0$. In such case,  we can pass to the limit    in \eqref{eq-existence-n} to get
\begin{equation}\label{eq-existence-0}
 L_{\lambda_0} v_0 =0\quad\text{ in }\quad\Omega,\qquad v_0=0 \  \text{ on }\ \R^N\setminus\Omega.
\end{equation}
We will discuss the consequences of \eqref{eq-existence-0} according to the values of $s$.

If $0<s<1/2,$ then as $v_0\in W^{2,p}(\Omega)$ for $1<p<\infty$, it follows from Lemma \ref{Uniqueness} that $v_0\equiv 0$, which yields again a contradiction. 

If $s=1/2,$ then we can apply the interior regularity result of Theorem \ref{Interior-regularity} to conclude that $v_0\in \cC^2(\Omega).$ The Sobolev embedding also tells us that $v_0\in C^{1,\alpha}(\overline{\Omega}).$ Thus, $v_0\in \cC^2(\Omega)\cap \cC(\R^N).$  The maximum principle in Theorem \ref{Strong-maximum-principle}  implies that $v_0\equiv 0,$ which is again a contradiction.

Therefore, for $s\in(0,1/2],$ the sequence $\{u_n\}_n$  remains uniformly bounded in $W^{2,p}(\Omega)$ and hence $\{C_n\}_n$ is bounded.  This is a contradiction and this completes the proof of Theorem \ref{strong-solution}. 
\end{proof}

  \begin{proof}[{\bf Proof of Theorem \ref{Holder1}}]
        We use the $L^p$ theory in Theorem \ref{strong-solution}. Indeed, since $ f \in C^{0,\alpha}(\overline{\Omega})$ and $\Omega$ is bounded, $ f \in L^p(\Omega)$ for any $p<\infty$. It follows then from Theorem \ref{strong-solution} that $ u \in C^{1,\beta}(\overline{\Omega})$ for any $\beta\in (0,1)$ (we choose all $p>N$). Take in particular $\beta=\alpha.$ From  the uniqueness of the solution, we get that 
    \begin{equation}\label{Estimate-1}
             \|u \|_{C^{1,\alpha}(\overline{\Omega})}\lesssim \|f\|_{{L^p}(\Omega)} \le C \|f\|_{C^{0,\alpha}(\overline{\Omega})}.
    \end{equation}
        As before, we can extend $u$ by zero outside  $\Omega$  by  a $C^{0,1}$  function in $\R^N$ and still  denote by $u$.   We apply again the regularity result of \cite[Proposition 2.5]{S07}  to see that $g:=(-\Delta )^s u \in C^{0,1-2s}(\R^N)$ with a control on the $\cC^{0,\alpha}$- norm of $g$ as follows
\begin{equation}\label{Estimate-2}
    \|g\|_{C^{0,\alpha}(\R^N)}\le C\|u \|_{C^{1,\alpha}(\overline{\Omega})}\le C\|f\|_{\cC^{1,\alpha}(\overline\Omega)}.
\end{equation}
Next, since $ q$ is H\"older over $\bar{\Omega}$, we have that $q\cdot\nabla u\in \cC^{0,\alpha}(\overline\Omega)$.  Now,  we  write the equation as $ -\Delta u = -g - q \cdot \nabla u $ in $\Omega$ with $u=0$ on $ \partial \Omega$. Hence, as $\Omega$ has smooth boundary, $u \in C^{2, \alpha}( \overline{\Omega})$ from the classical theory of elliptic PDEs. Moreover, combining \eqref{Estimate-1} and \eqref{Estimate-2}, there exists a constant $C>0$ such that
$$ \| u \|_{C^{2,\alpha}(\overline{\Omega})} \le C \|f\|_{C^{0,\alpha}(\overline{\Omega})}.$$ 
The proof of Theorem \ref{Holder1} is complete.
    \end{proof}

\section{Proof of Theorem \ref{Main-Result-1}}
 In this section we complete the proof  of Theorem \ref{Main-Result-1}. We recall the following statement of Krein-Rutman Theorem from \cite[Theorem 1.2]{D06} as we will use it in the proof of Theorem \ref{Main-Result-1}.

\begin{athm}[Krein-Rutman Theorem, \cite{D06}]\label{K-R}
Let $X$ be a Banach space,  $K\subset X$ a solid cone, $T:X\to X$ a compact linear operator which satisfies $T(K\setminus \{0\})\subset K^{\circ}$. Then, 
\begin{enumerate}[(i)]
    \item $r(T)>0$ and $r(T)$ is a simple eigenvalue with an eigenfunction $v\in K^{\circ}$; there is no other eigenvalue with a positive eigenfunction.
    \item   $|\mu|<r(T)$ for all eigenvalues $\mu\neq r(T)$.
\end{enumerate}
\end{athm}

\medskip

We now give the proof of Theorem \ref{Main-Result-1}.

\begin{proof}[Proof of Theorem \ref{Main-Result-1}]  \label{Proof-Main}

Define the space 
\[
X:= \big\{u\in \cC^{0,1}(\overline{\Omega}): \quad u=0 \text{ on }\partial\Omega \text{ and } u=0\ \text{ in }\ \R^N\setminus\Omega\big\}
\]
and the cone 
 \[
K:=\big\{ u\in X:\quad u\ge 0\ \text{ in }\ \overline\Omega\big\}.
\]
We will denote the interior of $K$ by $K^{\circ}.$  Indeed, \[K^\circ=\{ u \in X: \mbox{ there is some } \e>0 \mbox{ such that }  u(x) \ge \e \dist(x, \partial \Omega))\text{ for all } x \in \Omega \}.\] We now  define  the operator 
$$T:X\to X$$
by  $Tf=u,$ where $u$ is the solution of the problem \begin{equation}\label{eq-Regularity5}
	\begin{split}
	\quad\left\{\begin{aligned}
		L  u :=&-\Delta u  +(-\Delta)^{s}u+q\cdot\nabla u  &= f && \text{ in \quad $\Omega$,}\\
	u &=0     &&& \text{ on } \quad \R^N\setminus\Omega.
	\end{aligned}\right.
	\end{split}
	\end{equation} 
Clearly $T$ is a linear operator. The operator $T$ is bounded since, by Theorem \ref{strong-solution} and the Sobelev embedding, we have    $\|u\|_{\cC^{1,\alpha}(\overline{\Omega})}\leq C \|f\|_{\cC^{0,\alpha}(\overline{\Omega})}.$  Thus,   $$\|u\|_{\cC^{0,1}(\overline{\Omega})}\lesssim \|f\|_{\cC^{0,1}(\overline{\Omega})}.$$

Let us now prove that $T(K\setminus\{0\})\subseteq K^{\circ}.$ Let $f\in K$ such that $f\not\equiv 0$ and set $Tf=w.$ Hence, $Lw=f$ and $w$ satisfies \eqref{eq-Regularity5}. We separate two cases according to $s.$

\noindent 
{\bf Case 1. $0<s<1/2$:} in this case,  $w\in \cC^{2,\alpha}(\overline{\Omega})$ (Theorem \ref{Holder1}). We can apply the strong maximum principle (Theorem \ref{Strong-maximum-principle}) to the following problem
\[
Lw=f\ge 0 \quad\text{ in }\quad\Omega, \qquad w=0 \text{ on }\ \R^N\setminus\Omega.
\] 
We obtain that $w>0$ in ${\Omega}.$ Moreover, as $w=0$ on $\partial \Omega,$ it follows from Theorem \ref{Hopf} that $\partial_ \nu w(x)<0$ for all $x\in \partial \Omega.$

\noindent
{\bf Case 2. $s=1/2$:} in this case, we have the interior regularity of $w$ from Theorem \ref{Interior-regularity}. So we still have  that $w\in \cC^2(\Omega)\cap \cC(\R^n),$ as $w$ solves \eqref{eq-Regularity5} and  $w\in \cC^{1,\alpha}(\overline{\Omega})$ by the Sobolev embedding.  Applying the Hopf Lemma---stated in Theorem  \ref{Hopf2},  we obtain again here that $\partial_\nu w(x)<0$ for all $x\in \partial \Omega.$ 

Thus, for all $s\in (0,1/2]$, we have $\displaystyle{\partial_\nu w}(x)<0$ for all $x\in \partial \Omega.$ As $\partial \Omega$ is compact, then $\displaystyle{\max_{\partial \Omega}{\partial_{ \nu}}w(x)<0}.$ This allows us find an open $\cC^{0,1}$ neighbourhood $\mathcal{O}$ of $w,$ such that \[\mathcal{O}\subseteq \{ u \in X: \mbox{ there is some } \e>0 \mbox{ such that }  u(x) \ge \e \dist(x, \partial \Omega)), \forall  x \in \Omega \}\subseteq K^\circ.\] 
Thus, $w=Tf \in K^\circ.$

 
 We now verify that $T$ is compact. Let $\{f_n\}_n\subset X$ be a  bounded sequence in $X.$ Let us say that   $\|f_n\|_{L^{\infty}(\Omega)}\le 1$. It follows that $f_n\in L^p(\Omega)$ for any $1<p<\infty$  and from the $L^p$- theory in Theorem \ref{strong-solution}, we have that $Tf_n\in W^{2,p}(\Omega)$ for any $1<p<\infty$. The Sobolev embedding   implies that $Tf_n\in \cC^{1,\alpha}(\overline{\Omega})$  for any $0<\alpha<1$ and hence
 \[
 \|Tf_n\|_{\cC^{1,\alpha}(\overline{\Omega})}\le \|u\|_{W^{2,p}(\Omega)}\lesssim \|f_n\|_{L^{p}(\Omega)}\le C,
 \]
 where $C$ is a constant independent of $n.$
 This implies that $\{Tf_n\}_n$ is bounded in $\cC^{1,\alpha}(\overline{\Omega}).$ By the Arzela-Ascoli theorem, the sequence $\{Tf_n\}_n$ has a convergent subsequence (the convergence of the subsequence holds in $\cC^{1}(\overline{\Omega})$ and hence in $\cC^{0,1}(\overline\Omega)$).   This proves that $T$ is compact.

 Therefore, we can apply the  Krein-Rutman theorem to assert that  there exists a unique positive real number $\varrho(T)>0$ and a unique (up to multiplication by a nonzero constant) positive function $f\in K^{\circ}$ such that  $Tf=\varrho(T)f$. Therefore,  the function $\phi_1:=Tf>0$ satisfies the problem
	\begin{equation}\label{eq-proof}
	\begin{split}
	\quad\left\{\begin{aligned}
		L  \phi_1   &= \varrho(T)^{-1} \phi_1 && \text{ in \quad $\Omega$}\\
	\phi_1 &=0     && \text{ on } \quad \R^N\setminus\Omega.
	\end{aligned}\right.
	\end{split}
	\end{equation}
The function $\varphi_1$ is in $\cC^{2,\alpha}(\overline{\Omega})$ (from Theorem \ref{Holder1}). 

To complete the proof of Part (a) of the theorem, we set $\lambda_1(\Omega,q) :=\varrho(T)^{-1}.$ Then,  from the Krein-Rutman theorem (Theorem \ref{K-R}), $\lambda_1(\Omega,q)$ is the principal eigenvalue for $L$ in $\Omega$ with the corresponding unique (up to multiplication by a nonzero constant) positive eigenfunction given by $\phi_1\in \cC^{2,\alpha}(\overline{\Omega})$.

From part (ii) of Theorem \ref{K-R}, we know that  $\varrho=r(T)>0$ satisfies: any eigenvalue $\mu\neq \varrho$ for $T$ satisfies $|\mu|< \varrho.$ Now, since $\lambda_1(\Omega,q) = \varrho(T)^{-1}>0,$  the proof of part (b)  follows.

We are left to prove the max-inf formulation \eqref{minmax} of $\lambda_1(\Omega, q)$-- stated in part (c) of Theorem \ref{Main-Result-1}. 
We recall that $\cV(\Omega)$ is given by
	\[
\cV(\Omega):=\Big\{u\in \cC^2(\Omega)\cap\cC^1( \overline\Omega)\cap \cC_c(\R^N): \ u>0 \text{ in }{\Omega} \ \text{ and }\ u\equiv 0 \text{ on }\ \R^N\setminus\Omega\Big\}.
\]  
Since  $\phi_1\in \cV(\Omega)$, it follows that
\begin{equation*}
	\lambda_1(\Omega,q)\le \sup_{u\in \cV(\Omega)}\inf_{x\in  \Omega} \frac{Lu(x)}{u(x)}.
	\end{equation*}
 Thus, we only need to prove that 
 \begin{equation*}
	\lambda_1(\Omega,q)\ge \sup_{u\in \cV(\Omega)}\inf_{x\in  \Omega} \frac{Lu(x)}{u(x)}.
	\end{equation*} and once we have equality then one sees we can replace the sup with a max.  So we now argue by contradiction. Suppose that 
	\begin{equation*}
	\lambda_1(\Omega,q)< \sup_{u\in \cV(\Omega)}\inf_{x\in \Omega} \frac{Lu(x)}{u(x)}.
	\end{equation*}
	Then, there exists  $\varepsilon>0$ and  a function $v\in \cV(\Omega)$ such that 
	\begin{equation}\label{epsilon-def}
	\lambda_1(\Omega,q)+\varepsilon< \inf_{x\in \Omega} \frac{Lv(x)}{v(x)}.
	\end{equation}  Then note we have $Lv > (\lambda_1(\Omega,q)+\e)v $ in $ \Omega$ with $ v=0$ on $ \R^N \backslash \Omega$ and hence by Hopf's Lemma we have $ \partial_\nu v<0$.    
 We now define \begin{equation}\label{tau}
 \tau^*:=\sup\{ \tau>0 :\ v-\tau \phi_1\ge 0 \ \text{ in }\ \Omega\}
 \end{equation}
 and note $ 0<\tau^*<\infty$ after noting that $ \phi_1$ is sufficiently regular and the above Hopf result for $v$.  
 
 We now set $ w= v- \tau^* \phi_1$. First note that $ w \ge 0$ in $ \Omega$ and $ w=0$ on $\R^N\setminus\Omega$. Since $ \e>0$ we see that $ v$ cannot be a multiple of $ \phi_1$ and hence $ w$ is not identically zero.  Then note we have 
 \[ Lw  = Lv - \tau^* L\phi_1 > \e v + \lambda_1(\Omega, q) w\geq 0 \quad \mbox{ in } \Omega.\] 
 From the strong maximum principle we have $ w>0$ in $ \Omega$ or $w\equiv 0$ in $\R^N$. However, $w\equiv 0$ contradicts that $Lw>0.$ Now, from Hopf Lemma (Theorem \ref{Hopf2}), we know that $\partial_{\nu} v-\partial_\nu (\tau^* \phi_1)=\partial_\nu w<0$ on $\partial \Omega.$  Thus, as $v\geq \tau^*\phi_1\geq  0$ and $\partial_\nu (\tau^* \phi_1)>\partial _\nu v$ we can still find $\delta >0$ such that $v\geq (\tau^*+\delta)\phi_1\geq 0$ in $\Omega.$ This contradicts   the fact that $\tau^*$ is the largest possible  in \eqref{tau}. Therefore \eqref{epsilon-def} is false and we have 
\begin{equation}\label{sup}
	\lambda_1(\Omega,q)=\sup_{u\in \cV(\Omega)} \inf_{x\in \Omega} \frac{Lu(x)}{u(x)}.
	\end{equation}
Furthermore, we know from part (a) of this theorem that 
\[\phi_1\in \cV(\Omega)\text{ and }L\phi_1=\lambda_1(\Omega, q)\phi_1\text{ in }\Omega.\] Thus, the $\sup$ in \eqref{sup} is indeed a $\max$ that is attained at $\phi_1.$
This completes the proof of $(c)$ and hence the proof of Theorem \ref{Main-Result-1}.
\end{proof}

\bibliographystyle{amsplain}

\end{document}